\documentclass[10pt]{amsart}
\usepackage{algorithm,algorithmic,amsmath,amssymb,amsfonts,graphicx, hyperref,color}

\title{On the Turing model complexity of interior point methods for semidefinite programming}

\author{Etienne de Klerk}
\address{E.~de Klerk, Department of Econometrics and Operations Research, Faculty
  of Economic Sciences, Tilburg University,  5000 LE Tilburg, The Netherlands}
\email{E.deKlerk@uvt.nl}

\author{Frank Vallentin}
\address{F.~Vallentin, Mathematisches Institut, Universit\"at zu
  K\"oln, Weyertal~86--90, 50931 K\"oln, Germany}
\email{frank.vallentin@uni-koeln.de}
\thanks{The second author was partially supported by VIDI grant 639.032.917 from the Netherlands Organization for Scientific Research (NWO)}

\date{July 16, 2015}

\newcommand{\R}{\mathbb{R}}

\newtheorem{defin}{Definition}[section]
\newtheorem{definition}[defin]{Definition}

\newtheorem{theorem}[defin]{Theorem}

\newtheorem{corollary}[defin]{Corollary}
\newtheorem{lemma}[defin]{Lemma}

\DeclareMathOperator{\val}{val}

\DeclareMathOperator{\trace}{{\rm Trace}}
\DeclareMathOperator{\size}{{\rm size}}
\DeclareMathOperator{\sym}{{\rm sym}}

\begin{document}

\maketitle

\markboth{E.~de Klerk, F.~Vallentin}{On the Turing model complexity of interior
point methods for semidefinite programs}

\begin{abstract}
  It is known that one can solve semidefinite programs to within fixed
  accuracy in polynomial time using the ellipsoid method (under some
  assumptions). In this paper it is shown that the same holds true
  when one uses the short-step, primal interior point method.  The
  main idea of the proof is to employ Diophantine approximation at
  each iteration to bound the intermediate bit-sizes of iterates.
\end{abstract}

\section{Introduction}
\label{sec:introduction}

Semidefinite programming is used in several polynomial-time
algorithms, like the celebrated Goemans-Williamson \cite{goe95}
approximation algorithm for the maximum cut problem, the algorithm for
computing the stability number of a perfect graph
\cite{GroetschelLovaszSchrijver1981a}, and many others (see e.g.\
\cite{Gartner2012}).  To give a rigorous proof of the polynomial-time
complexity of such algorithms, one requires a known theorem, due to
Gr\"otschel, Lov\'asz, and Schrijver
\cite{GroetschelLovaszSchrijver1981a}, on the Turing model complexity
of solving semidefinite programs to fixed precision (under some
assumptions).  In \cite{GroetschelLovaszSchrijver1981a}, this theorem
is proved constructively by using the ellipsoid method of Yudin and
Nemirovski \cite{Yudin-Nemirovski 1976} (inspired by the earlier proof
of Khachiyan \cite{Khacijan1} of the polynomial-time solvability of
linear programming), but our aim here is to do so by using the theory
of interior point methods.  Perhaps surprisingly, such a proof has not
yet been given to the best of the authors' knowledge.

For example, in Chapter 2 of the recent book \cite{Gartner2012} it is
stated that:
\begin{quote} [...] the ellipsoid method is the only known method that
  provably yields polynomial runtime [for semidefinite programming] in
  the Turing machine model [...]
\end{quote}

The complexity theorem in question may be stated as follows.

\begin{theorem}[Gr\"otschel, Lov\'asz, Schrijver \cite{GroetschelLovaszSchrijver1981a}]
\label{th:sdppolynomialtime}
Consider the semidefinite program
\begin{equation}
\label{eq:sdp}
\begin{array}{rl}
\val = \inf & \langle C,X \rangle\\
& X \in \mathcal{S}^n \text{ is positive semidefinite},\\
& \langle A_j, X \rangle = b_j \text{ for } j = 1, \ldots, m,\\
\end{array}
\end{equation}
with rational input $C$, $A_1, \ldots, A_m$, and $b_1, \ldots, b_m$,
and where $\mathcal{S}^n$ denotes the set of $n\times n$ symmetric
matrices. Denote by
\[
\mathcal{F} = \{X \in \mathcal{S}^n : X \text{ is positive
  semidefinite, }  \langle A_j, X \rangle = b_j \text{ for } j = 1, \ldots, m\}
\]
the set of feasible solutions. Suppose we know a rational point
$X_0 \in \mathcal{F}$ and positive rational numbers $r$, $R$ so that
\[
X_0 + B(X_0,r) \subseteq \mathcal{F} \subseteq X_0 + B(X_0,R),
\]
where $B(X_0,r)$ is the ball of radius $r$, centered at $X_0$, in the $d$-dimensional subspace
\[
L = \{X \in \mathcal{S}^n : \langle A_j, X \rangle= 0 \text{ for } j = 1,
\ldots, m\}.
\]
For every positive rational number $\epsilon>0$ one can find in
polynomial time a rational matrix $X^* \in \mathcal{F}$ such that
\[
\langle C, X^* \rangle - \val \leq \epsilon,
\]
where the polynomial is in $n$, $m$, $\log_2 \frac{R}{r}$,
$\log_2(1/\epsilon)$, and the bit size of the data $X_0$, $C$,
$A_1, \ldots, A_m$, and $b_1, \ldots, b_m$.
\end{theorem}

Here $\langle X, Y \rangle = \trace(XY)$ denotes the trace inner
product for symmetric matrices, and hence, when we talk about the
ball $B(X_0,r)$ or $B(X_0,R)$ we work with the associated Frobenius norm
\[
\|X\|_F = \langle X, X \rangle^{1/2}.
\]

We will show that the analysis by Renegar \cite{Renegar2001a} of the
short step interior point algorithm, together with applying
Diophantine approximation at every step to ensure that the bit size
stays small, leads to a proof of Theorem \ref{th:sdppolynomialtime}.

There is also a practical aspect to the results in this paper.
Semidefinite programming is increasingly used in computer-assisted
proofs.  Thus new theoretical results have been obtained in this way
for binary code sizes \cite{Schrijver04}, crossing numbers of graphs
\cite{KlePasSch:02}, binary sphere packings \cite{Vallentin2014}, and
other problems.  To obtain rigorous proofs, it is necessary to give a
formal verification of the relevant semidefinite programming bound.
Usually this is done by computing dual bounds using floating point
arithmetic, and then showing rigorously that the corresponding dual
solutions are feasible. This type of ``reverse engineering'' can be
quite cumbersome; see e.g.\ the discussion in \cite[Section 5.3]{Vallentin2014}.
Moreover, the semidefinite programs involved are often numerically
ill-conditioned, and it may be difficult or impossible to obtain a
near-optimal solution with off-the-shelf solvers; see e.g.\
\cite{Mittelmann2010}. It is therefore of practical interest to
understand what may be done in polynomial time when using exact
arithmetic. We note that there already exists an arbitrary precision
solver, SDPA-GMP (see \cite{SDPA} and the references therein) that
uses the GNU multi-precision linear algebra library.  The algorithmic
ideas presented here may potentially be used to enhance such a solver
to improve its performance, by ensuring that it runs in polynomial
time, i.e.\ that the intermediate bit-sizes do not become excessively
large.

Finally, one should note that there have been several papers studying the
complexity of interior point methods using {\em finite precision
  arithmetic} (allowing only a fixed number of bits for calculations);
see e.g.\ \cite{Vera,Wright,Gu}.  For the Turing model complexity
though, the only results known to us concern interior point methods
for linear programming; see e.g.\ the original paper by Karmarkar
\cite{int:Karmarkar2}, or the review in the book of Wright
\cite{int:Swright}.

\section{Preliminaries}

In this section we set up the notation for the paper. Since we follow
Renegar's proof we mainly use his notation.

\subsection{SDP problem structure and notation}

\begin{itemize}
\item
We will denote matrices (and matrix variables) by capital letters, and general vectors (or variables) by lower case letters.
\item By $\mathcal{S}^n$ we denote the $\binom{n+1}{2}$-dimensional
  vector space of symmetric matrices which is endowed with the trace inner
  product $\langle X, Y \rangle = \trace(XY)$. The corresponding norm
  is the Frobenius norm
\[
\|X\|_F = \langle X, X \rangle^{1/2} = \sum_{i=1}^n \lambda_i(X)^2,
\]
where $\lambda_i(X)$ is the $i$-th largest eigenvalue of the symmetric
matrix $X$. By $\mathcal{S}^n_{\succeq 0}$ we denote the closed convex cone of
positive semidefinite matrices, and $\mathcal{S}^n_{\succ 0}$ is the
open cone of positive definite matrices.
If the matrix size is clear from the context, we will sometimes write $X \succ 0$ (resp.\ $X\succeq 0$) instead of
$X \in \mathcal{S}^n_{\succ 0}$ (resp.\ $X \in \mathcal{S}^n_{\succeq 0}$).
\item The semidefinite program \eqref{eq:sdp} defines the linear
  operator $A \colon \mathcal{S}^n \to \R^m$ componentwise by
\[
(AX)_j = \langle A_j, X \rangle, \quad \text{with} \quad j = 1,
\ldots, m.
\]
Its adjoint operator $A^* \colon \R^m \to \mathcal{S}^n$ is
\[
A^* y = \sum_{j=1}^m y_j A_j,
\]
where we take the adjoint with respect to the trace inner
product. From now on we assume that $A$ is surjective. Hence, the
adjoint $A^*$ is injective, and the matrices $A_1, \ldots, A_m$ are
linearly independent.

The kernel of $A$ is the linear subspace
\[
L = \ker A = \{X \in \mathcal{S}^n : AX = 0\}
\]
and the matrices $A_1, \ldots, A_m$ form a basis of the orthogonal
complement~$L^\perp$. The orthogonal projection onto the subspace
$L$ is given by
\[
\pi_L = I_{\mathcal{S}^n} - A^*(A A^*)^{-1} A,
\]
where $I_{\mathcal{S}^n}$ is the identity operator for ${\mathcal{S}^n}$.
\item
We may (and will) assume that $C \in L$, without loss of generality.
Indeed, every feasible $X \in F$ may be written as $X = X_0 + \Delta X$ for some $\Delta X \in L$, so that
\begin{eqnarray*}
\langle C,X\rangle &=& \langle C,X_0\rangle + \langle C,\Delta X\rangle \\
&=& \langle C-\pi_L(C)+\pi_L(C),X_0\rangle + \langle \pi_L(C),\Delta X\rangle \\
&=&  \langle C-\pi_L(C),X_0\rangle + \langle \pi_L(C), X\rangle.
\end{eqnarray*}
Thus we may replace $C$ by $\pi_L(C)$ if necessary.
Moreover, the bit-size of $\pi_L(C)$ is bounded by a  polynomial in the bit-size of $C$ and $A$, due to Theorem \ref{th:linalg} below.
\end{itemize}

\subsection{Polynomial-time operations }

For ease of reference, we will use the framework in the book of
Schrijver \cite{Schrijver LP ILP book} when discussing complexity. In
particular, we use the same definition for the bit-size of rational
numbers, vectors and matrices as in \cite[\S2.1]{Schrijver LP ILP
  book}, and we will denote bit-size by $\size(\cdot)$.  In
particular, for relatively prime $p,q \in \mathbb{Z}$, we define the
bit-size of the rational number $p/q$ as:
\[
\size(p/q) = 1 + \lceil\log_2 |p| + 1\rceil + \lceil\log_2 |q| + 1\rceil.
\]
The bit size of a rational vector $(p_1/q_1,\ldots,p_n/q_n)$ is
defined as the sum of the bit sizes of its components plus $n$.
Similarly, the bit size of an $m\times n$ matrix is defined as the sum
of the bit sizes of its components plus $m\times n$.

\subsubsection*{Diophantine approximation}

We will perform a ``rounding'' procedure at the end of each iteration to
reduce the bit-size of the iterate, and will use {\em Diophantine
  approximation} for this.

\begin{theorem}[cf.\ Corollary 6.2a in \cite{Schrijver LP ILP book}]
\label{th:Euclid}
Let $\alpha$ and $0 < \epsilon \le 1$ be given rational numbers.
Then one may find, in time polynomial in the bit size of $\alpha$, integers $p$ and $q$ such that
\[
\left| \alpha - \frac{p}{q}\right| < \frac{\epsilon}{q} \mbox{ and } 1 \le q \le \frac{1}{\epsilon}, \; |p| \le \lceil |\alpha| \rceil q.
\]
\end{theorem}

The underlying algorithm is the {\em continued fraction method}; see
page 64 in \cite{Schrijver LP ILP book} for a description of the
algorithm.

As an immediate corollary, one may approximate a rational vector
$\alpha \in \mathbb{Q}^n$ componentwise by a rational vector
$(p_1/q_1,\ldots,p_n/q_n)$ such that
\begin{equation}
\label{eq:Diophantine}
\left\| (\alpha_1,\ldots,\alpha_n) - \left(\frac{p_1}{q_1},\ldots,\frac{p_n}{q_n}\right)\right\|_2 < \epsilon\sum_{i=1}^n\frac{1}{ q_i}, \; \forall i:\;
 1 \le q_i \le \frac{1}{\epsilon}, \; |p_i| \le \lceil |\alpha_i| \rceil q_i,
\end{equation}
in time polynomial in the bit-size of the vector $\alpha$.

We restate this result in a form that we will need later.

\begin{corollary}
\label{cor:Diophantine}
Given a rational vector $\alpha \in \mathbb{Q}^n$ and rational $\epsilon >0$,
one may compute in  time polynomial in $\size(\alpha)$ integers $p_1,\ldots,p_n$ and $q_1,\ldots,q_n$
such that
\begin{equation}
\label{eq:Diophantine2}
\left\| (\alpha_1,\ldots,\alpha_n) - \left(\frac{p_1}{q_1},\ldots,\frac{p_n}{q_n}\right)\right\|_2 < \epsilon,
\end{equation}
such that
\[
\size(p_1/q_1,\ldots,p_n/q_n) \le n\left(6+\log_2 \left(\frac{n^2\lceil \|\alpha\|_\infty \rceil}{\epsilon^2}\right)\right).
\]
\end{corollary}
\proof
Assume the integers $p_i,q_i$ ($i \in \{1,\ldots,n\}$)  satisfy \eqref{eq:Diophantine}.
For each $i$ one has
\[
|p_i| \le  \lceil |\alpha_i|\rceil q_i \le \lceil \|\alpha\|_\infty \rceil q_i \le \lceil \|\alpha\|_\infty \rceil \frac{1}{\epsilon}.
\]
Thus
\begin{eqnarray*}
\size(p_i/q_i) &=& 1 + \lceil\log_2 |p_i| + 1\rceil + \lceil\log_2 |q_i| + 1\rceil \\
&\le& 1 + \left\lceil\log_2  \frac{ \lceil \|\alpha\|_\infty \rceil}{\epsilon} + 1\right\rceil + \left\lceil\log_2 \frac{1}{\epsilon} + 1\right\rceil \\
&\le & 5 + \log_2  \frac{ \lceil \|\alpha\|_\infty \rceil}{\epsilon^2}.
\end{eqnarray*}
As a consequence
\[
\size(p_1/q_1,\ldots,p_n/q_n) = n + \sum_{i=1}^n \size(p_i/q_i) \le n\left(6+\log_2 \left(\frac{\lceil \|\alpha\|_\infty \rceil}{\epsilon^2}\right)\right).
\]
Using \eqref{eq:Diophantine}, $\sum_{i=1}^n\frac{1}{ q_i} \le n$, and replacing $\epsilon$ by $\epsilon/n$  completes the proof. \qed

\subsubsection*{Linear algebra}

Each iteration of the short-step interior point algorithm involves
some linear algebra operations, and we will use the following results
to ensure that this may be done in polynomial time.

\begin{theorem}
\label{th:linalg}
The following operations on matrices may be performed in polynomial time (in the bit sizes of the matrices and vectors):
\begin{enumerate}
\item
Matrix addition and multiplication;
\item
Matrix inversion;
\item
Solving linear systems with Gaussian elimination;
\item
Computing an orthogonal basis (using Gaussian elimination and Gram-Schmidt orthogonalization) of a nullspace $\{x \; : \; Ax = 0\}$ where
the rational matrix $A$ is given.
\end{enumerate}
\end{theorem}

For a proof, see e.g.\  Theorem 3.3 and Corollary 3.3a in \cite{Schrijver LP ILP book}.

The last item implies that we may compute an orthogonal basis for $L$
(the nullspace of $A$), so that we may represent any feasible point
$X \in \mathcal{F}$ as $X = X_0 + \sum_{i=1}^d x_i B_i$, say, where
the $x_i$ are scalars and the $B_i$'s are suitable symmetric matrices
of size polynomial in the input size that form an orthogonal basis for
$L$.  We may also assume without loss of generality that
$\|B_i\|_F \le 1$ for each $i$. This is important, since we will study
perturbations (roundings) of the form $\bar X = X + \Delta X$, where
$X \in \mathcal{F}$ and $\Delta X \in L$.  Writing
$\Delta X = \sum_{i=1}^d \Delta x_i B_i$, one then has
$\|\Delta X\|_F \le \|\Delta x\|_2$. In other words, we may bound the
size of the perturbation in $\mathcal{S}^n$ in terms of the
corresponding perturbation in $\mathbb{R}^d$.

\subsection{Self-concordant barrier functions}

We will use the definition of self-con\-cordant functions due to Renegar
\cite{Renegar2001a}, that is more suited to our purposes than the
original definition of Nesterov and Nemirovski \cite{int:Nesterov5}.
In what follows, $f$ is a convex functional with open convex domain
$D_f$ (contained in a finite-dimensional, real affine space), and the
gradient and Hessian of $f$ at $x \in D_f$ will be denoted by $g(x)$
and $H(x)$ respectively. Note that the gradient and Hessian depend on
the inner product we choose for the underlying vector space; see \S1.2
and \S1.3 in Renegar \cite{Renegar2001a} for more details.

\begin{definition}[cf.\ \S 2.2.1 in \cite{Renegar2001a}]
Assume $f:D_f \to \mathbb{R}$ (with $D_f$ open and convex) is such that $H(x) \succ 0$ for all $x \in D_f$.
Then $f$ is called self-concordant  if:
\begin{enumerate}
\item
For all $x \in D_f$ one has $B_x(x,1) \subseteq D_f$;
\item
For all $y \in B_x(x,1)$ one has
\[
1 - \|y-x\|_x \le \frac{\|v\|_y}{\|v\|_x} \le \frac{1}{1 - \|y-x\|_x} \mbox{  for all } v \neq 0,
\]
\end{enumerate}
where $\| v \|_x := \langle v, H(x)v\rangle^{\frac{1}{2}}$ is called the intrinsic (or local) norm of $v$, and $B_x(x,1)$ is the unit ball, centered at $x$,
with respect to the intrinsic norm.
\end{definition}

A self-concordant functional $f$ is called a \emph{self-concordant barrier} if there is a finite value $\vartheta_f$ so that
\[
\vartheta_f = \sup_{x \in D_f} \| H^{-1}g(x)\|_x,
\]
that is, the intrinsic norm (at $x$) of the Newton step
$n(x) := -H(x)^{-1}g(x)$ is always upper bounded by $\vartheta_f$. The
{\em analytic center} of $D_f$ is defined as the (unique) minimizer of
$f$. (The analytic center exists if and only if $D_f$ is bounded.)

The self-concordant barrier function of the semidefinite program
\eqref{eq:sdp} is
\begin{equation}
\label{eq:SDP barrier}
f(X) = -\ln \det X \mbox{ with domain }
D_f = \mathcal{S}^n_{\succ 0} \cap \{X \in \mathcal{S}^n : AX = b\}.
\end{equation}
For this barrier function one has $\vartheta_f \le n$; see \cite[\S 2.3.1]{Renegar2001a}.
Its gradient  (with respect to the trace inner product) is
\[
g(X) = -\pi_L(X^{-1}),
\]
and its Hessian is
\[
H(X)Y = \pi_L(X^{-1} Y X^{-1}) \quad \text{with} \quad Y \in L.
\]
The local norm for $Y \in L$ at $X \in D_f$ is defined as
\[
\|Y\|_X = \langle Y, H(X) Y\rangle^{1/2}.
\]
For easy reference, we note that the self-concordance of  the function $f$ in (\ref{eq:SDP barrier}) implies that for all $X \in D_f$ we have
$B_X(X,1) \subseteq D_f$ and that for all $Y \in B_X(X,1)$ we have
\begin{equation}
\label{eq:self-concordance}
1 - \|Y-X\|_X \leq \frac{\|V\|_Y}{\|V\|_X} \leq \frac{1}{1-\|Y -
  X\|_X} \quad \text{for all $V \in L \setminus \{0\}$,}
\end{equation}
where $B_X(Y,r)$ denotes the open ball of radius $r$ centered at $Y$
in the local norm $\|\cdot\|_X$. %, see \cite[Chapter 2.2.1]{Renegar2001a}.

\subsubsection{Properties of self-concordant functions}

We will need the following three technical results (and one corollary) on self-concordant functions.

\begin{theorem}[Theorem 2.2.3 in \cite {Renegar2001a}]
\label{th:2.2.3}
Assume $f $ self-concordant and $x \in D_f$.  If $z$ minimizes $f$ and $z \in B_x(x,1)$ then
\[
x^+ := x - H(x)^{-1}g(x)
\]
satisfies
\[
\|x^+ - z \|_x \le \frac{\|x-z\|_x^2}{1-\|x-z\|_x}.
\]
\end{theorem}

A useful, and immediate, corollary is the following.

\begin{corollary}
\label{cor:2.2.3}
Under the assumptions of Theorem \ref{th:2.2.3}, one has
\[
\|n(x)\|_x := \|H(x)^{-1}g(x)\|_x \le \frac{\|x-z\|_x}{1-\|x-z\|_x}.
\]
\end{corollary}
\proof
By definition,
\begin{eqnarray*}
\|n(x)\|_x  &=& \|x^+ - x \|_x \\
&\le & \|x^+ - z \|_x + \|z  - x \|_x \\
&\le &  \frac{\|x-z\|_x^2}{1-\|x-z\|_x} +  \|x  - z \|_x \quad \mbox{(by Theorem \ref{th:2.2.3})} \\
&=& \frac{\|x-z\|_x}{1-\|x-z\|_x},
\end{eqnarray*}
as required. \qed

The other two technical results are the following.

\begin{theorem}[Theorem 2.2.4 in \cite {Renegar2001a}]
\label{th:2.2.4}
Assume $f $ self-concordant and $x \in D_f$ such that $\|n(x)\|_x \le 1$.
 Then
\[
\|n(x^+)\|_{x^+} \le \left(\frac{\|n(x)\|_x}{1- \|n(x)\|_x}\right)^2.
\]
\end{theorem}

\begin{theorem}[Theorem 2.2.5 in \cite {Renegar2001a}]
\label{th:2.2.5}
Assume $f $ self-concordant and $x \in D_f$ such that $\|n(x)\|_x \le 1/4$.
 Then  $f$ has a minimizer $z$ and
\[
\|z-x^+\|_x \le \frac{3\|n(x)\|_x^2}{(1- \|n(x)\|_x)^3}.
\]
Thus (triangle inequality):
\[
\|x - z \|_x \le \|n(x)\|_x + \frac{3\|n(x)\|_x^2}{(1- \|n(x)\|_x)^3}.
\]
\end{theorem}

\section{The short-step, logarithmic barrier algorithm}

We consider a generalisation of our SDP problem, given by
\[
\val := \min_{x \in \mbox{cl}( D_f)} \langle c,x \rangle,
\]
where $c$ is a given vector, $f$ is a self-concordant barrier with
open domain $D_f$, and cl$(D_f)$ denotes the closure of $D_f$.  As
before, the gradient and Hessian of $f$ at $x \in D_f$ are
respectively denoted by $g(x)$ and $H(x)$.

For the SDP problem (\ref{eq:sdp}), $f(X) = - \ln \det (X)$ with
domain $D_f = \{X \succ 0 \; : \; X \in \mathcal{F}\}$, but Algorithm
\ref{alg1} below is valid for a general self-concordant barrier.

Define, for given $\eta > 0$,
\[
f_\eta(x) := \eta \langle c,x \rangle +  f(x),
\]
and denote by  $n_\eta(x) = -H(x)^{-1}(\eta c+ g(x))$ the (projected) Newton direction at $x$ for $f_\eta$.

The analytic curve, parameterized by $\eta > 0$, where $\eta$ is
mapped to the unique minimizer of $f_\eta$, is called the {\em central
  path}.

\begin{algorithm}                      % enter the algorithm environment
\caption{Short step algorithm}          % give the algorithm a caption
\label{alg1}                           % and a label for \ref{} commands later in the document
\begin{algorithmic}                    % enter the algorithmic environment
    \REQUIRE  an $x_1\in D_f$ and $\eta_1 >0$ such that
$\|n_{\eta_1}(x_1)\|_{x_1} \le \frac{1}{4}$.
     An accuracy parameter $\epsilon > 0$.
    \STATE $k \leftarrow 1$
     \WHILE{$\frac{\vartheta_f}{\eta_k} > \epsilon$}
     \STATE Set $x_{k+1} = x_k + n_{\eta_k}(x_k)$
     \STATE Set $\eta_{k+1} = \left(1+ \frac{1}{8\sqrt{\vartheta_f}}\right)\eta_k$
\STATE $k \leftarrow k+1$.
%
%        \IF{$N$ is even}
%            \STATE $X \Leftarrow X \times X$
%            \STATE $N \Leftarrow N / 2$
%        \ELSE[$N$ is odd]
%            \STATE $y \Leftarrow y \times X$
%            \STATE $N \Leftarrow N - 1$
%        \ENDIF
    \ENDWHILE
\end{algorithmic}
\end{algorithm}

The complexity of the short step algorithm is described in the
following theorem, that is originally due to Nesterov and Nemirovski
\cite{int:Nesterov5}.

\begin{theorem}[cf.\ p.\ 47 in \cite{Renegar2001a}]
\label{thm:short step complexity}
The short step algorithm terminates after at most
\[
k = \left\lceil 10\sqrt{\vartheta_f}\ln\left(\frac{7\vartheta_f}{6\eta_1\epsilon}\right) \right\rceil
\]
iterations.
The output is a feasible point $x_k$ such that
\[
\langle c,x_k\rangle - \val  \le  \epsilon.
\]
\end{theorem}

Some remarks on the steps in the algorithm.
\begin{itemize}
\item
For the SDP problem (\ref{eq:sdp}), the projected Newton direction is obtained by first solving the following linear system:
\begin{equation}
\label{eq:Newton1}
My = v
\end{equation}
where
\[
M_{ij} = \trace(XA_iXA_j), \; (i,j \in \{1,\ldots,m\})
\]
and
\[
v_i = -b_i + \eta \trace(A_iXCX), \; (i \in \{1,\ldots,m\}).
\]
(We drop the subscript $k$ that refers to the iteration number here for convenience.)
Subsequently, the projected Newton direction is given by
\begin{equation}
\label{eq:Newton2}
n_{\eta}(X) = X(A^*y)X + X - \eta XCX.
\end{equation}
The matrix $M$ is positive definite (and hence nonsingular) under the assumption that $\{A_1,\ldots,A_m\}$ are linearly independent.
One may bound the sizes of $M$ and $v$ in (\ref{eq:Newton1}) as follows:
\begin{eqnarray*}
\size(M_{ij}) &\le & \size(XA_i) + \size(A_jX) \\
                    &\le & n(\size(X)+\size(A_i))+ n(\size(X)+\size(A_j)),
\end{eqnarray*}
so that
\[
\size(M) \le m^2(1 + 2n\size(X)) + 2mn\sum_{i=1}^m \size(A_i).
\]
Similarly,
\[
\size(v) \le m + 2mn\size(X) + mn\size(C) + 2n\sum_{i=1}^m \size(A_i)+\size(b)+m\size(\eta).
\]

As a consequence, the projected Newton direction may be computed in
time polynomial in the bit sizes of $X$, $\eta$ and the data $A$, $b$
and $C$.  Thus one may perform a constant number of iterations in
polynomial time.  We will show how to truncate the current iterate $X$
at the end of each iteration, using Diophantine approximation, in
order to guarantee that the bit-size of the iterates remains suitably
bounded throughout.
\item
The square root $\sqrt{\vartheta_f}$ that appears in the statement of the algorithm may be replaced by any larger number,
 e.g. $\lceil \sqrt{\vartheta_f} \rceil$. The only change to the complexity is that $\sqrt{\vartheta_f}$
 should then be replaced by the corresponding larger value in the statement of Theorem \ref{thm:short step complexity}.
 \item
 By construction, each iterate $x_k$ satisfies $\|n_{\eta_k}(x_k)\|_{x_k} \le \frac{1}{4}$, and after the Newton step one therefore has
 \begin{equation}
 \label{eq:proximity}
 \|n_{\eta_k}(x_{k+1})\|_{x_{k+1}} \le \frac{1}{9},
 \end{equation}
  by Theorem \ref{th:2.2.4}. As a result, after setting $\eta_{k+1} = \left(1+ \frac{1}{8\sqrt{\vartheta_f}}\right)\eta_k$,
 one again has $\|n_{\eta_{k+1}}(x_{k+1})\|_{x_{k+1}} \le \frac{1}{4}$; see \cite[p. 46]{Renegar2001a} for details.
 Since we will apply rounding (using Diophantine approximation) to the iterates later on, we will need to ensure that (\ref{eq:proximity}) still holds after rounding $x_{k+1}$.
\item
An issue that needs to be resolved is the initialization question, i.e.\ finding
$x_1\in D_f$ and $\eta_1 >0$ (of suitable bit size) such that
$\|n_{\eta_1}(x_1)\|_{x_1} \le \frac{1}{4}$. This is addressed in the next section.
\end{itemize}

\section{Initialization}

Assume now --- again in the setting of a general self-concordant barrier $f$ --- that we only know a rational starting point $x' \in D_f$.
We will use a two phase procedure, where we first solve an auxiliary problem to obtain a suitable starting point for the short step algorithm.
The procedure here follows Renegar \cite[\S2.4]{Renegar2001a}.

\subsection*{Auxiliary problem}

For a given parameter $\nu > 0$, we consider the auxiliary problem where we minimize:
\[
f'_\nu(x) := - \nu \langle g(x'),x\rangle + f(x).
\]
Note that $x'$ is {on the central path} of the auxiliary problem and corresponds to $\nu = 1$.

Now use the short step algorithm, {\em reducing} $\nu$ at each iteration via
\[
\nu_{k+1} = \left(1- \frac{1}{8\sqrt{\vartheta_f}}\right)\nu_k.
\]
Remarks:
\begin{itemize}
\item
The central path of the auxiliary problem passes through $x'$ and {converges to
the analytic center of $D_f$} as $\nu \downarrow 0$.
\item
Once $\nu$ is small enough, we may use the current value of $x$
as {a starting point for the original short step algorithm}.
\item
After
 \[
k \ge 10\sqrt{\vartheta_f}\ln\left(\frac{7}{6\epsilon'}\right),
\]
iterations, we have $\nu_k \le \epsilon'$, by Theorem \ref{thm:short step complexity}.
\item
In the SDP case of problem (\ref{eq:sdp}), one has $x' = X_0$ and $g(x') = -\pi_L(X_0^{-1})$, that has bit-size polynomial in the input size, by Theorem \ref{th:linalg}.
\item
A suitable choice for $\epsilon'$ that provides a starting point for the second phase depends on the (Minkowski) symmetry of $D_f$ around $x'$.
\end{itemize}

\begin{definition}[Symmetry of $D$ around $x$]
Let $D$ be a bounded open convex set and $x \in D$.
Let ${\mathcal{L}}(x,D)$ denote the set of lines that pass through $x$.
For any $\ell \in {\mathcal{L}}(x,D)$, let $r(\ell)$ denote the {ratio of the shorter to
the longer line segments} $\ell \cap (D \setminus \{x\})$. Finally define the symmetry of $D$ around $x$ as
\[
\sym(x,D) := \inf_{\ell \in {\mathcal{L}}(x,D)} r(\ell).
\]
\end{definition}

A suitable value for $\epsilon'$ is now given by
\begin{equation}
\label{eq:epsilon prime}
\epsilon' =  \frac{1}{18\vartheta_f (1+1/\sym(x',D_f))}.
\end{equation}
At this point one may start the short step algorithm using $x_1$ equal to the last iterate produced by solving the auxiliary problem, and
\begin{equation}
\eta_1 = \frac{1}{12\|H(x_1)^{-1}c\|_{x_1}} \ge \frac{1}{12}\left(\sup_{x \in D_f} \langle c,x\rangle - \val\right).
\label{eq:eta one}
\end{equation}
See \S2.4 in \cite{Renegar2001a} for more details and proofs.

The combined complexity of this two-phase procedure is given by the
following theorem.  The proof is easily extracted from the proof of
Theorem 2.4.1 in \cite{Renegar2001a}.

\begin{theorem}[{\em cf.} Theorem 2.4.1 in \cite{Renegar2001a}]
\label{thm:2.4.1 in Renegar}
Assume $f \in \mathcal{SCB}$ and $D_f$ bounded. Assume a starting point $x' \in D_f$.
If $0 < \epsilon < 1$, then within
\[
10\sqrt{\vartheta_f}\ln\left(\frac{294\vartheta_f^2}{\epsilon}\left(\frac{1}{1+ \sym(x',D_f)}\right)\right)
\]
iterations, all points $x$ computed thereafter satisfy
\[
\langle c,x\rangle - \val \le
\epsilon\left(\sup_{x \in D_f} \langle c,x\rangle - \val\right).
\]
\end{theorem}

For the SDP problem (\ref{eq:sdp}) we now assume, as in Theorem \ref{th:sdppolynomialtime}, that we have a rational $X_0 \in \mathcal{F}$, and
 that we know rational $r>0$ and $R >0$ so that  $X_0+B(X_0,r) \subset \mathcal{F} \subset X_0 + B(X_0,R)$.
Note that this implies:
\begin{equation}
\label{eq:SDP starting point symmetry}
\sym(X_0,\mathcal{F}) \ge \frac{r}{R}.
\end{equation}

\section{An upper bound on the norm of the dual central path}

In this section we give an upper bound on the norm of the dual central
path.  Our analysis is based on a standard argument for the existence
and uniqueness of the central path; see e.g.\ \cite[Proof of Theorem
10.2.1]{Monteiro2000a}.

Recall that the \emph{(primal-dual) central path} is the curve
$\eta \mapsto (X(\eta), S(\eta), y(\eta))$, with $\eta > 0$, defined
as  the unique solution of
\[
AX = b, \; A^*y + S = C, \; XS = \frac{1}{\eta}I, \; X \succ 0, \; S \succ 0,
\]
where $I$ denotes the identity matrix.
\begin{lemma}
\label{lemma:dual path bound}
Under the assumptions stated in Theorem~\ref{th:sdppolynomialtime}
we have
\begin{equation}
\label{S_eta_bound}
\|S(\eta)\|_F \leq \frac{\sqrt{n}}{(1-1/e)r} \left(\langle X_0, C + 2\|C\|_\infty I \rangle + \frac{n}{r \eta^2}\right),
\end{equation}
where
\[
\|C\|_\infty = \max_{i=1,\ldots,n} \sum_{j=1}^n |C_{ij}|
\]
is the maximum row sum norm of $C$.
\end{lemma}

\begin{proof}
By assumption $X_0$ is a strictly feasible solution of the primal and
without loss of generality we may assume that $S_0 = C + 2\|C\|_\infty I$ is a strictly
feasible solution of the dual; otherwise we add the constraint
\[
\langle I, X \rangle \leq \langle I, X_0 \rangle + \sqrt{n} R
\]
to the semidefinite program~\eqref{eq:sdp} which is redundant since
\[
\langle I, X-X_0 \rangle \leq \left(\langle I, I \rangle \langle X-X_0, X -
X_0 \rangle\right)^{1/2} \leq \sqrt{n} R.
\]
Note that $S_0$ is indeed positive definite, since it is strictly diagonally dominant.

We may characterize $S(\eta)$ as the unique minimizer of the function
\[
S \mapsto \langle X_0, S \rangle - \frac{1}{\eta} \ln\det S
\]
over the set $\{S : S = C -A^*y, \; S \succ 0, \; y \in \R^m\}$.

As in \cite[Proof of Theorem 10.2.1]{Monteiro2000a}, we define the set
\[
\begin{split}
\mathcal{U} = \Big\{S : \; & S = C -A^*y, \; S \succ 0, \; y \in \mathbb{R}^m,\\
& \quad \langle X_0, S \rangle - \frac{1}{\eta}\ln\det S \leq \langle X_0 ,
S_0 \rangle - \frac{1}{\eta}\ln\det S_0\Big\}.
\end{split}
\]
Clearly, $\mathcal{U}$ contains $S(\eta)$.

If $\sigma > 0$ denotes the smallest eigenvalue of $X_0$, then, for all $S \in \mathcal{U}$:
\[
\sigma \langle I, S \rangle - \frac{1}{\eta}\ln\det S \leq \langle
X_0, S_0 \rangle - \frac{1}{\eta}\ln\det S_0,
\]
because $\sigma \langle I, S\rangle \leq \langle X_0, S\rangle$. Now
we write the previous inequality in terms of the eigenvalues
$\lambda_i(S)$ of $S$:
\[
\sum_{i=1}^n \left(\sigma \lambda_i(S) - \frac{1}{\eta}\ln
  \lambda_i(S) \right) \leq \langle X_0, S_0 \rangle -
\frac{1}{\eta}\ln\det S_0.
\]

Defining the function
\[
\phi(\lambda) = \sigma \lambda  - \frac{1}{\eta}\ln \lambda,  \quad
\text{for } \lambda > 0,
\]
which is convex and has minimizer $\lambda^* = \frac{1}{\sigma\eta}$
with minimum value
$\phi(\lambda^*) = \frac{1}{\eta}\left( 1 - \ln \frac{1}{\sigma\eta}
\right)$, one has
\[
\phi\left(\lambda_i(S)  \right) \le \langle X_0, S_0 \rangle -
\frac{1}{\eta}\ln\det S_0 - (n-1)\phi(\lambda^*) \quad \text{for }
i = 1, \ldots, n.
\]
By the convexity of $\phi$ and by approximating $\phi$ about the point
$e\lambda^*$ we have
\[
\phi(\lambda) \geq \phi(e\lambda^*) + \phi'(e\lambda^*)(\lambda -
e\lambda^*) = (1-1/e) \sigma \lambda -
\frac{1}{\eta} \ln \frac{1}{\sigma \eta}.
\]
Hence,
\[
\begin{split}
\lambda_i(S) \; & \leq \; \frac{1}{(1-1/e)\sigma} \left(\langle X_0, S_0 \rangle -
\frac{1}{\eta} \ln\det S_0 - (n-1) \phi(\lambda^*) + \frac{1}{\eta}
\ln \frac{1}{\sigma \eta}\right)\\
& \leq \; \frac{1}{(1-1/e)\sigma} \left(\langle X_0, S_0 \rangle -
\frac{2n-1}{\eta}  + \frac{n}{\sigma \eta^2}\right)\\
& \leq \; \frac{1}{(1-1/e)r} \left(\langle X_0, S_0 \rangle + \frac{n}{r \eta^2}\right) \quad \text{for }
i = 1, \ldots, n,
\end{split}
\]
where the first inequality follows from $\det S_0 \geq 1$ and $\ln x
\leq x - 1$, and where the second inequality follows because $\sigma
\geq r$.

The last estimate now immediately implies the statement of the lemma:
\[
\|S(\eta)\|_F \leq \frac{\sqrt{n}}{(1-1/e)r} \left(\langle X_0, S_0 \rangle + \frac{n}{r \eta^2}\right).
\]
\end{proof}

Note that the bound on $\|S(\eta)\|_F$ depends on the value of $\eta$.
It is therefore necessary to consider the range of values that $\eta$ can take (and $\nu$ during the first phase of the auxiliary problem).
During the first phase (auxiliary problem),
initially $\nu_1 = 1$, which is subsequently decreased via $\nu_{k+1} = \left(1- \frac{1}{8\sqrt{\vartheta_f}}\right)\nu_k$. It is simple to show that
during each iteration $k$ of the first phase,
\[
1 \ge \nu_k \ge {\epsilon'},
\]
where $\epsilon'$ is defined in (\ref{eq:epsilon prime}), which in turn implies
\begin{equation}
\label{eq:nu bounds}
1 \ge \nu_k \ge \frac{1}{18n(1+ R/r)},
\end{equation}
where we have used (\ref{eq:epsilon prime}) and (\ref{eq:SDP starting point symmetry}).

Similarly, during each iteration $k$ of the second phase
\[
\frac{1}{12}\left(\sup_{X \in D_f} \langle C,X\rangle - \val\right) \le \eta_k \le \frac{\vartheta_f}{\epsilon},
\]
which implies
\begin{equation}
\label{eq:eta bounds}
\frac{1}{6}r\|C\|_F \le \eta_k \le \frac{n}{\epsilon},
\end{equation}
since $B(X_0,r) \subseteq \mathcal{F}$ and $\vartheta_f \le n$.

\section{Rounding the current iterate}

We will round the current iterate $X \in \mathcal{F}$ (we again drop
the subscript for convenience) at the end of each iteration to obtain
a feasible $\bar X = X + \Delta X$, say, with suitably bounded
bit-size, and where the "rounding error" $\Delta X \in L$ satisfies
$\|\Delta X\|_X \le \tilde\epsilon$ for some suitable value
$\tilde\epsilon > 0$.

After the Newton step, but before the update of $\eta$, we assume that
\[
\left\| X - X(\eta) \right\|_X \le c'
\]
where $c'>0$ is a known constant.

By the definition of self-concordance:
\begin{eqnarray*}
\|\bar X - X(\eta) \|_{\bar X} &\le & \frac{1}{1 - \|\Delta X\|_X}\left\| X + \Delta X - X(\eta)\right\|_X \\
                          &\le & \frac{1}{1-\tilde\epsilon}\left\| X + \Delta X - X(\eta) \right\|_X \\
                          &\le & \frac{1}{1-\tilde\epsilon}\left\| X - X(\eta) \right\|_X +  \frac{1}{1-\tilde\epsilon}\left\| \Delta X  \right\|_X\\
                          &\le & \frac{c' + \tilde\epsilon}{1-\tilde\epsilon}.
\end{eqnarray*}
Thus, if $\tilde\epsilon = \frac{1}{16}$, and $c' = \frac{1}{32}$ then $\|\bar X - X(\eta) \|_{\bar X} \le  \frac{1}{10}$.
Consequently, by Corollary \ref{cor:2.2.3}, one has $\|n_\eta(\bar X)\|_{\bar X}  \le \frac{1}{9}$, as required (recall \eqref{eq:proximity}).

We may ensure that $\|X - X(\eta) \|_X \le \frac{1}{32}$ during the course of the algorithm by taking an extra centering step.
Indeed, if we still denote the iterate by $X$ after an extra centering step,
 one has $\|n_\eta(X)\|_X \le 1/64$ (by Theorem \ref{th:2.2.4}). Consequently,
by Theorem \ref{th:2.2.5}, one has
\[
\|X - X(\eta)\|_X \le \|n_\eta(X)\|_X + \frac{3\|n_\eta(X)\|^2_X}{(1-\|n_\eta(X)\|_X)^3} < \frac{1}{32}.
\]
Note that $\bar X \succ 0$ since $\|X - \bar X \|_X \le \frac{1}{16} < 1$, and the definition of self-concordance guarantees
that the unit ball in the $X$-norm centered at $X$ is contained
in the positive definite cone.

The task is therefore to find $\bar X = X + \Delta X$ with bounded
bit-size and so that $\|\Delta X\|_X \le \frac{1}{16}$.

It will be more convenient to bound the $X(\eta)$-norm of $\Delta X$
than the $X$-norm.  As a first observation, using the definition of
self-concordance,
\begin{eqnarray*}
\| X-X(\eta)\|_{X(\eta)} &\le&
  \frac{\|X-X(\eta)\|_X}{1-\|X(\eta) -
  X\|_X} \\
&\le&  \frac{\|X-X(\eta)\|_X}{1-\frac{1}{32}} \\
&=&  \frac{32}{31} \|X-X(\eta)\|_{X}.
\end{eqnarray*}
Invoking the definition of self-concordancy once more, we obtain:
\begin{eqnarray*}
\|\Delta X\|_X &\le&
  \frac{\|\Delta X\|_{X(\eta)}}{1-\|X(\eta) -
  X\|_{X(\eta)}} \\
  &\le&
  \frac{\|\Delta X\|_{X(\eta)}}{1-\frac{32}{31}\|X(\eta) -
  X\|_{X}} \\
&\le& \frac{\|\Delta X\|_{X(\eta)}}{1-\frac{32}{31}\cdot\frac{1}{32}} \\
&=&  \frac{31}{30} \|\Delta X\|_{X(\eta)}.
\end{eqnarray*}
Thus if we show that $\|\Delta X\|_{X(\eta)} \le \frac{30}{31 \times 16}$ then we guarantee that $\|\Delta X\|_X \le \frac{1}{16}$.

Note that
\begin{eqnarray*}
\|\Delta X\|^2_{X(\eta)} & \le & \langle \Delta X, X(\eta)^{-1}\Delta X X(\eta)^{-1}\rangle \\
                        &=& \eta^2\langle \Delta X, S(\eta)\Delta X S(\eta)\rangle \\
                        &\le & \eta^2 \|\Delta X\|_F^2 \|S(\eta)\|_F^2, \\
\end{eqnarray*}
where the inner product is the Euclidean (trace) inner product, and we have used the sub-multiplicativity of the Frobenius norm.

Recall that $\|S(\eta)\|_F$ is bounded by (\ref{S_eta_bound}) (Lemma \ref{lemma:dual path bound}).

We may now use Diophantine approximation so that
\begin{equation}
\label{eq:delta X bound 1}
\|\Delta X\|_F \le \frac{30}{31 \times 16}\left(\nu\frac{\sqrt{n}}{(1-1/e)r} \left(\langle X_0, -\pi_L(X_0^{-1}) + 2\|\pi_L(X_0^{-1})\|_\infty I \rangle + \frac{n}{r \nu^2}\right)\right)^{-1},
\end{equation}
during the first phase of the algorithm, and
\begin{equation}
\label{eq:delta X bound 2}
\|\Delta X\|_F \le \frac{30}{31 \times 16}\left(\eta\frac{\sqrt{n}}{(1-1/e)r} \left(\langle X', C + 2\|C\|_\infty I \rangle + \frac{n}{r \eta^2}\right)\right)^{-1},
\end{equation}
during the second phase, where $X'$ is the last iterate produced by the first phase.

Due to the upper and lower bounds on $\nu$ in (\ref{eq:nu bounds}), \eqref{eq:delta X bound 1} will hold if
$\|\Delta X\|_F \le \epsilon_1$, where
\[
\frac{1}{\epsilon_1} :=  \frac{17\sqrt{n}}{(1-1/e)r} \left(\langle X_0, -\pi_L(X_0^{-1}) + 2\|\pi_L(X_0^{-1})\|_\infty I \rangle + \frac{n(18n(1+ R/r))^2}{r} \right),
\]
during the first phase, and \eqref{eq:delta X bound 2} will hold if, during the second phase,
\[
\|\Delta X\|_F \le
\left(\frac{17(\sqrt{n})^3}{(1-1/e)r\epsilon} \left(\langle X', C + 2\|C\|_\infty I \rangle + \frac{36n}{r^3 \|C\|_F^2}\right)\right)^{-1}.
\]
To obtain a right-hand-side expression in terms of the input data only, we may use $\|X'-X_0\|_F \le R$. Thus we find that the last inequality will hold if
$\|\Delta X\|_F \le  \epsilon_2$, where
\[
\frac{1}{\epsilon_2} := \frac{17(\sqrt{n})^3}{(1-1/e)r\epsilon} \left( (R+\|X_0\|_F)\|C + 2\|C\|_\infty I\|_F + \frac{36n}{r^3 \|C\|_F^2}\right).
\]

Setting $\bar \epsilon = \min\{\epsilon_1,\epsilon_2\}$, implies that $\log_2 \left(\frac{1}{\bar \epsilon}\right)$ is bounded by a polynomial in the input size.

Performing Diophantine approximation in the $d$-dimensional space $L$
yields a rational $\bar X$ so that $\|\Delta X\|_F \le \bar \epsilon$
and
\begin{equation}
\label{eq:size X bound}
\size(\bar X) \le d\left(6+\log_2 \left(\frac{d^2\lceil R \rceil}{{\bar\epsilon}^2}\right)\right),
\end{equation}
by Corollary \ref{cor:Diophantine}.

Thus the size of $\bar X$ is always bounded by a certain polynomial in the input size.

\section{Summary and conclusion}

To summarize, we list the complete procedure in Algorithm \ref{alg2}.
The main subroutine (used twice) is a short step algorithm with extra
centering step and Diophantine approximation, shown as Algorithm
\ref{alg0}.

\begin{algorithm}                      % enter the algorithm environment
\caption{Short step algorithm with extra centering and  Diophantine approximation}          % give the algorithm a caption
\label{alg0}                           % and a label for \ref{} commands later in the document
\begin{algorithmic}                    % enter the algorithmic environment
    \REQUIRE    $\;$
    \begin{itemize}
    \item Problem data $(A,b,c)$;
    \item an $x_1\in D_f$ and $\eta_1 >0$ such that
$\|n_{\eta_1}(x_1)\|_{x_1} \le \frac{1}{4}$;
\item
     an accuracy parameter $\varepsilon > 0$;
     \item an update parameter $\theta > 0$;
      \end{itemize}
    \STATE $k \leftarrow 1$
     \WHILE{$\frac{(1-\theta)}{\eta_k} > (1-\theta)\varepsilon$}
     \STATE Set $x^+ = x_k + n_{\eta_k}(x_k)$
     \STATE Set $x_{k+1} = x^+ + n_{\eta_k}(x^+)$
     \STATE Round $x_{k+1}$ using Diophantine approximation, so that $\size(x_{k+1})$ is bounded as in \eqref{eq:size X bound}, and $\|n_{\eta_k}(x_{k+1})\|_{x_{k+1}} \le \frac{1}{9}$
     \STATE Set $\eta_{k+1} = \theta \cdot \eta_k$
\STATE $k \leftarrow k+1$
    \ENDWHILE
\end{algorithmic}
\end{algorithm}

\begin{algorithm} [h!]                     % enter the algorithm environment
\caption{Two-phase short step algorithm with Diophantine approximation}          % give the algorithm a caption
\label{alg2}                           % and a label for \ref{} commands later in the document
\begin{algorithmic}                    % enter the algorithmic environment
     \REQUIRE    $\;$
    \begin{itemize}
    \item SDP problem data $(A,b,c)$ and $X_0\in \mathcal{F}$;
    \item
     an accuracy parameter $\epsilon > 0$;
                    \item
          rational $R > r >0$ as in Theorem \ref{th:sdppolynomialtime}.
     \end{itemize}
      \STATE \emph{{\em \bf First phase (auxiliary problem):}}
    \STATE Set $c = -\pi_L(X_0^{-1})$, $\eta_1 = 1$, $x_1 = X_0$, $\varepsilon =  \frac{1}{18\vartheta_f (1+R/r)}$, $\theta = 1+\frac{1}{8\sqrt{\vartheta_f}}$ \STATE Call Algorithm \ref{alg0} with input $(A,b,c, x_1, \eta_1,\varepsilon,\theta)$
   \STATE \emph{{\em \bf Second phase:}
}   \STATE Set $c = C$, $\eta_1$ as in (\ref{eq:eta one}), $x_1$ equal to the last iterate of the first phase, $\varepsilon =  \epsilon/\vartheta_f$, $\theta = 1-\frac{1}{8\sqrt{\vartheta_f}}$
       \STATE Call Algorithm \ref{alg0} with input $(A,b,c, x_1, \eta_1,\varepsilon,\theta)$
\end{algorithmic}
\end{algorithm}

In particular, we have shown the following.

\begin{theorem}
Under the assumptions of Theorem \ref{th:sdppolynomialtime}, Algorithm \ref{alg2}
computes in polynomial time a rational matrix $X^* \in \mathcal{F}$ such that
\[
\langle C,X^*\rangle - \val \le
\epsilon\left(\max_{X \in \mathcal{F}} \langle C,X\rangle - \val\right),
\]
where the polynomial is in $n$, $m$, $\log_2 r$, $\log_2 R$,
$\log_2(1/\epsilon)$, and the bit size of the data $X_0$, $C$,
$A_1, \ldots, A_m$, and $b_1, \ldots, b_m$.
\end{theorem}

The analysis presented here may also be performed for more practical
variants of the interior point method, such as the long-step (large
update) method; see e.g.\ Chapter 2 in \cite{Renegar2001a}. Moreover,
since all computations in Algorithm \ref{alg2} involve linear algebra
only (Diophantine approximation may also be implemented as such),
there are definite practical perspectives for implementing Algorithm
\ref{alg2} (or a more practical variant), using arbitrary precision
packages, like the GNU Multiple Precision Arithmetic Library (GMP)
(\url{https://gmplib.org/}), that is already used in the solver
SDPA-GMP \cite{SDPA}.

\end{document}